\documentclass[12pt]{article}
\usepackage{fullpage}
\usepackage{amsmath,amssymb,amsfonts,amsthm}
\usepackage{enumerate}
\usepackage{epsfig}
\usepackage[all,poly]{xy}
\usepackage{colortbl}

\usepackage{geometry}
\usepackage{tikz}
\usepackage{mathrsfs}

\newif\ifnotesw\noteswtrue% T to show box & marginal notes; F suppresses.
   {\ifnotesw\marginpar[\hfill\(\top\)]{\(\top\)}\fi}%
   {\ifnotesw\marginpar[\hfill\(\bot\)]{\(\bot\)}\fi}

\newcommand{\mnote}[1]%
    {\ifnotesw\marginpar%
        [{\scriptsize\begin{minipage}[t]{\marginparwidth}
        \raggedleft#1%
                        \end{minipage}}]%
        {\scriptsize\begin{minipage}[t]{\marginparwidth}
        \raggedright#1%
                        \end{minipage}}%
    \fi}
\newcommand{\ignore}[1]{}

\newsavebox{\given}
\savebox{\given}[1em]{\rule[-1.5ex]{.2mm}{4ex}}

\newtheorem{theorem}{Theorem}
\newtheorem{corollary}[theorem]{Corollary}
\newtheorem{lemma}[theorem]{Lemma}

\newtheorem{fact}[theorem]{Fact}

\newcommand{\iverson}[1]{\lbrack\!\lbrack #1 \rbrack\!\rbrack}
\newcommand{\zo}{\{0,1\}}

\newcommand{\CC}{\mathbb{C}}

\newcommand{\GG}{\mathcal{G}}

\newcommand{\SC}{\mathcal{C}}

\newcommand{\ket}[1]{\mathbf{e}(#1)}
\newcommand{\bra}[1]{\ket{#1}^{\dagger}}
\newcommand{\braket}[2]{\bra{#1}\ket{#2}}
\newcommand{\ketbra}[2]{\ket{#1}\bra{#2}}

\newcommand{\cart}{\mbox{ $\Box$ }}

\DeclareMathOperator{\Alt}{\mathfrak{A}}
\DeclareMathOperator{\Ker}{ker}
\DeclareMathOperator{\sgn}{sgn}

\newcommand{\wedgep}[2]{\bigwedge^{#2} #1}
\newcommand{\wedgepo}[3]{\bigwedge^{#2}_{#3} #1}

\newcommand{\xwedgep}[2]{\mbox{$\wedgep{#1}{#2}$}}
\newcommand{\xwedgepo}[3]{\mbox{$\wedgepo{#1}{#2}{#3}$}}
\newcommand{\veep}[2]{{#1}^{\{ #2 \}}}

\newcommand{\symdif}{\triangle}

\title{
Which Exterior Powers are Balanced?
}
\author{
Devlin Mallory\footnote{Department of Mathematics, University of California at Berkeley.}
\and
Abigail Raz\footnote{Department of Mathematics, Wellesley College.}
\and
Christino Tamon\footnote{Department of Computer Science, Clarkson University. Contact author: tino@clarkson.edu}
\and
Thomas Zaslavsky\footnote{Department of Mathematics, Binghamton University.}
}
\date{\today}
\begin{document}
\maketitle
\bibliographystyle{plain}

\begin{abstract}
A signed graph is a graph whose edges are given $\pm 1$ weights.
In such a graph, the sign of a cycle is the product of the signs of its edges.
A signed graph is called balanced if its adjacency matrix is similar to
the adjacency matrix of an unsigned graph via conjugation by a diagonal $\pm 1$ matrix.
For a signed graph $\Sigma$ on $n$ vertices, its exterior $k$th power,
where $k=1,\ldots,n-1$, is a graph $\bigwedge^{k} {\Sigma}$ whose adjacency matrix is given by
\[ A(\mbox{$\bigwedge^{k} {\Sigma}$}) = P_{\wedge}^{\dagger} A(\Sigma^{\Box k}) P_{\wedge}, \]
where $P_{\wedge}$ is the projector onto the anti-symmetric subspace of the
$k$-fold tensor product space $(\mathbb{C}^{n})^{\otimes k}$ and $\Sigma^{\Box k}$ is the $k$-fold
Cartesian product of $\Sigma$ with itself.
The exterior power creates a signed graph from any graph, even unsigned.
We prove sufficient and necessary conditions so that $\bigwedge^{k} {\Sigma}$ is balanced.
For $k=1,\ldots,n-2$, the condition is that either $\Sigma$ is a signed path 
or $\Sigma$ is a signed cycle that is balanced for odd $k$ or is unbalanced for even $k$;
for $k=n-1$, the condition is that each even cycle in $\Sigma$ is positive and each odd cycle
in $\Sigma$ is negative.

\vspace{.1in}
\end{abstract}

%%%%%%%%%%%%%%%%%%%%%%%%%%%%%%%%%%%%%%%%%%%%%%%%%%%%%%%%%%%%%%%%%%%%%%%%%%%%%%%%%%%%%%%%%%%%%%%%%%%%%%%%%%%%%%%%%%

\section{Introduction}

We are interested in graph operators which arise from taking the quotient of a 
Cartesian product of an underlying graph with itself. 
More specifically, such operators are defined on a graph $G=(V,E)$ after applying
the following three steps.
First, we take the $k$-fold Cartesian product of $G$ with itself, namely $G^{\Box k}$.
Note that the vertex set of $G^{\Box k}$ is the set of $k$-tuples $V^{k}$.
For the second (possibly optional) step, we remove from $V^{k}$ (via vertex deletions) 
the set $\mathcal{D}$ consisting of all $k$-tuples of vertices which contain a repeated vertex. 
We denote the resulting graph as $G^{\Box k} \setminus \mathcal{D}$.
Finally, we conjugate the adjacency matrix of $G^{\Box k} \setminus \mathcal{D}$ 
with an invertible operator $\mathcal{Q}$. 
If we call the resulting graph $\GG$, then
\begin{equation} \label{eqn:quotient}
A(\GG) = \mathcal{Q}^{\dagger} A(G^{\Box k} \setminus \mathcal{D}) \mathcal{Q}.
\end{equation}
For cases of interest, $\mathcal{Q}$ will either be the symmetrizer $P_{\vee}$
or the anti-symmetrizer $P_{\wedge}$ which are projection operators onto 
the symmetric and anti-symmetric subspaces of the $k$-fold tensor product space 
$V^{\otimes k}$ (see Bhatia \cite{bhatia}). 

For example, Audenaert, Godsil, Royle and Rudolph \cite{agrr07} studied the 
{\em symmetric powers} $\veep{G}{k}$ of a graph $G$ which is defined as
\begin{equation}
A(\veep{G}{k}) = \tilde{P}^{\dagger} A(G^{\Box k} \setminus \mathcal{D}) \tilde{P},
\end{equation}
where $\tilde{P}$ is a ``hybrid'' of $P_{\vee}$ and $P_{\wedge}$.
These graphs were studied in the context of the graph isomorphism problem on strongly regular graphs. 
They are equivalent to the so-called $k$-tuple vertex graphs, 
introduced by Zhu, Liu, Dick and Alavi \cite{zlla92},
which generalizes the well-known double vertex graphs 
(see Alavi, Behzad, Erd\"{o}s, and Lick \cite{abel91}).
Subsequently, these $k$-tuple vertex graphs were studied by Osborne \cite{o06} 
in the framework of spin networks in quantum information.

The main focus of this work is on the {\em exterior powers} $\wedgep{G}{k}$ of a graph $G$ 
defined as
\begin{equation}
A(\xwedgep{G}{k}) = P_{\wedge}^{\dagger} A(G^{\Box k}) P_{\wedge}.
\end{equation}
We will show that $\wedgep{G}{k}$ is a {\em signed} graph whose edges have $\pm 1$ weights.
Moreover, this also holds if the underlying graph is a signed graph $\Sigma = (G,\sigma)$
where $\sigma$ is the $\pm 1$-valued edge-signing function.

Our main observation is the following characterization of balanced exterior powers.

\begin{theorem} \label{thm:signed-balanced}
Let $\Sigma$ be a signed graph on $n$ vertices. Then, $\wedgep{\Sigma}{k}$ is balanced if and only if:
\begin{itemize}
\item for $k = 2,\ldots,n-2$:\\
	Either (i) $\Sigma$ is a signed path; or 
	(ii) $\Sigma$ is a signed cycle that is balanced when $k$ is odd $k$ and unbalanced when $k$ is even.
\item for $k=n-1$:\\
	Each even cycle in $\Sigma$ is positive and each odd cycle is negative.
\end{itemize}
\end{theorem}

The proof of the first part of the above theorem is based on a forbidden subgraph property.
Namely, if the exterior $k$th power $\wedgep{\Sigma}{k}$ is balanced, then 
$\Sigma$ may not contain the claw $K_{1,3}$ as a subgraph, for $k = 1,\ldots,n-2$. 
In this case, $\Sigma$ is either a signed path or it is a signed cycle which 
is balanced for odd $k$ and unbalanced for even $k$. 
The case when $k = n-1$ requires a separate argument.

As a corollary, we also obtain a characterization of {\em unsigned} graphs $G$ 
whose exterior $k$th power is balanced. For $2 \le k \le n-2$, $\wedgep{G}{k}$ is balanced 
if and only if $G$ is a path or $G$ is a cycle {\em and} $k$ is odd. For the boundary
case of $k=n-1$, $\wedgep{G}{n-1}$ is balanced if and only if $G$ is bipartite.

We remark that both the symmetric and exterior powers are closely related to
the study of many-particle quantum walk on graphs in quantum information (see \cite{bgmrt}).

%%%%%%%%%%%%%%%%%%%%%%%%%%%%%%%%%%%%%%%%%%%%%%%%%%%%%%%%%%%%%%%%%%%%%%%%%%%%%%%%%%%%%%%%%%%%%%%%%%%%%%%%%%%%%%%%%%

\section{Preliminaries}

We state some notation which we use in the remainder of this paper.
For a logical statement $S$, we let $\iverson{S}$ be $1$ if $S$ is true, and $0$ otherwise.
Given a positive integer $n$, we use $[n]$ to denote the set $\{1,2,\ldots,n\}$.
For two sets $A$ and $B$, we let $A \symdif B$ denote the symmetric difference of $A$ and $B$;
that is, $A \symdif B = (A \setminus B) \cup (B \setminus A)$.

The graphs $G$ we study here are finite, mostly simple, undirected and connected. 
The vertex set of $G$ will be denoted $V(G)$ and its edge set $E(G)$.
The adjacency matrix $A(G)$ of $G$ is defined as $A(G)_{u,v} = \iverson{(u,v) \in E(G)}$.
For two graphs $G$ and $H$ with adjacency matrices $A(G)$ and $A(H)$, respectively,
their {\em Cartesian product} $G \cart H$ is a graph defined on the vertex set 
$V(G) \times V(H)$ where $(g_{1},h_{1})$ is adjacent to $(g_{2},h_{2})$ if either
$g_{1}=g_{2}$ and $(h_{1},h_{2}) \in E_{H}$, or $(g_{1},g_{2}) \in E_{G}$ and $h_{1}=h_{2}$.
The adjacency matrix of this Cartesian product is given by
$A(G \cart H) = A(G) \otimes I + I \otimes A(H)$.
Here, $A \otimes B$ denotes the tensor product of matrices $A$ and $B$.

Standard graphs we consider include 
the complete graphs $K_{n}$, paths $P_{n}$, cycles $C_{n}$, bipartite graphs, and the hypercubes $Q_{n}$.

A vertex partition $\pi$ of a graph $G=(V,E)$ given by $V = \biguplus_{j=1}^{m} V_{j}$ is
{\em equitable} if each vertex $u$ in $V_{j}$ is adjacent to $d_{j,k}$ vertices in $V_{k}$,
where this constant is independent of the choice of vertex $u$ (see Godsil \cite{godsil-dm11}).
Each component $V_{j}$ is called a {\em cell} of the equitable partition $\pi$.
Here, $|\pi| = m$ is the size of $\pi$ -- which is the number of cells in $\pi$.
The (normalized) partition matrix $Q$ of $\pi$ is a $|V| \times |\pi|$ matrix whose entries 
are $Q[i,j] = \iverson{i \in V_{j}}/\sqrt{|V_{j}|}$. 
The {\em quotient} graph $G / \pi$ of $G$ modulo the equitable partition $\pi$ is an
undirected weighted graph whose vertices are the cells of $\pi$ and whose edges $(V_{j},V_{k})$ 
are given the weights $\sqrt{d_{j,k}d_{k,j}}$.

\begin{fact} (Godsil \cite{godsil-dm11}) \label{fact:equitable-partition}
Let $G=(V,E)$ be a graph and let $\pi$ be an equitable partition with partition matrix $Q$.
The following properties hold:
\begin{enumerate}
\item $Q^{T} Q = I$.
\item $Q Q^{T}$ commutes with $A(G)$.
\item $A(G / \pi) = Q^{T} A(G) Q$.
\end{enumerate}
\end{fact}

\par\noindent
More background on algebraic graph theory may be found in Godsil and Royle \cite{godsil-royle01}.

\subsection{Signed graphs}
A {\em signed} graph $\Sigma = (G,\sigma)$ is a pair consisting of a graph $G=(V,E)$ and
a signing map $\sigma: E(G) \rightarrow \{-1,+1\}$ over the edges of $G$.
We call $G$ the underlying (unsigned) graph of $\Sigma$; 
we also use $|\Sigma|$ to denote this underlying graph.

The sign of a cycle in a signed graph is the product of all the signs of its edges. 
A signed graph is called {\em balanced} if all of its cycles are positive.
Similarly, it is called {\em anti-balanced} if all of its odd cycles are negative
and all of its even cycles are positive.

\begin{fact} (Harary \cite{h53})
A signed graph is balanced if and only if 
there is a bipartition in which all crossing edges are negative 
and all non-crossing edges are positive.
\end{fact}

Given a signed graph $\Sigma$, {\em switching} around a vertex $u \in V(\Sigma)$ means
flipping the signs of all edges adjacent to $u$. Similarly, switching around a subset 
$U \subseteq V(\Sigma)$ means flipping the signs of all edges crossing the cut $(U,U^{c})$,
that is, edges of the form $(v,w)$, where $v \in U$ and $w \in U^{c}$.

\begin{fact} 
A signed graph is balanced (respectively, anti-balanced) if and only if 
there is a subset $U$ for which switching around $U$ results in all edges being positive 
(respectively, negative). 
\end{fact}

Switching has a nice algebraic description as conjugation by a diagonal matrix with $\pm 1$ entries.
Two signed graphs $\Sigma_{1}$ and $\Sigma_{2}$ are called {\em switching equivalent}, 
denoted $\Sigma_{1} \sim \Sigma_{2}$, if $A(\Sigma_{1}) = D^{-1}A(\Sigma_{2})D$ for some diagonal
matrix $D$ with $\pm 1$ entries.
Thus, a signed graph $\Sigma$ is {\em balanced} if $\Sigma \sim |\Sigma|$
and it is {\em anti-balanced} if $-\Sigma \sim |\Sigma|$ (that is, $-\Sigma$ is balanced).

More basic facts about signed graphs may be found in Zaslavsky \cite{z82}.

%%%%%%%%%%%%%%%%%%%%%%%%%%%%%%%%%%%%%%%%%%%%%%%%%%%%%%%%%%%%%%%%%%%%%%%%%%%%%%%%%%%%%%%%%%%%%%%%%%%%%%%%%%%%%%%%%%

\section{Exterior Powers}

Let $G$ be a graph on $n$ vertices.
We first define the exterior $k$th power $\wedgep{G}{k}$, for $1 \le k \le n-1$, in 
a combinatorial fashion. An equivalent algebraic formulation will follow thereafter.
Our treatment follows closely the description given by Osborne \cite{o06} (see also \cite{agrr07}).

We fix an arbitrary total ordering $\prec$ on the vertex set $V$ of $G$.
Consider a $k$-tuple $u \in V^{k}$, say $u = (u_{1},\ldots,u_{k})$, whose elements 
are distinct and are ordered with respect to $\prec$, that is, $u_{1} \prec \ldots \prec u_{k}$.
We adopt the somewhat standard ``wedge'' product notation to denote such an ordered $k$-tuple
as $u = u_{1} \wedge \ldots \wedge u_{k}$.
Let $\binom{V}{k}$ denote the set of all such ordered $k$-tuples. That is:
\begin{equation}
\binom{V}{k} = \left\{u_{1} \wedge \ldots \wedge u_{k} : 
	u_{1},\ldots,u_{k} \in V, u_{1} \prec \ldots \prec u_{k}\right\}.
\end{equation}
Again, we stress that $u_{1} \wedge u_{2} \wedge \ldots \wedge u_{k}$ is notation 
for the $k$-tuple $(u_{1},\ldots,u_{k})$ whose elements are distinct and are given
in ascending order (with respect to $\prec$).

The vertex set of $\wedgep{G}{k}$ is $\binom{V}{k}$ and its edges are defined so
two vertices $u = \bigwedge_{j=1}^{k} u_{j}$ and $v = \bigwedge_{j=1}^{k} v_{j}$ 
are adjacent if there is a bijection $\pi$ over $[k]$ where $u_{\pi(j)} = v_{j}$ 
for all $j \in \{1,\ldots,k\}$ except at one index $i$ where $(u_{\pi(i)},v_{i}) \in E(G)$. 
Thus, for $u,v \in \binom{V}{k}$, we have
\begin{equation} \label{eqn:wedge-adjacency}
(u,v) \in E(\xwedgep{G}{k})
\ \ \
\mbox{ iff } 
\ \ \ 
(\exists i)[(u_{\pi(i)},v_{i}) \in E(G) \mbox{ and } (\forall j \neq i)[u_{\pi(j)} = v_{j}]]. 
\end{equation}
Here, we say that the permutation $\pi$ {\em connects} the vertices $u$ and $v$.
We make $\wedgep{G}{k}$ into a {\em signed} graph by assigning the value $\sgn(\pi)$ to the above edge $(u,v)$.

\vspace{.2in}
\par\noindent{\bf Example}: 
Consider a $4$-cycle on $V = \{a,b,c,d\}$ with edges $E = \{(a,b),(b,c),(c,d),(a,d)\}$. 
With the ordering $a \prec b \prec c \prec d$, we have
$\binom{V}{2} = \{a \wedge b, a \wedge c, a \wedge d, b \wedge c, b \wedge d, c \wedge d\}$.
Then, $\wedgep{C_{4}}{2}$ is the signed bipartite graph $K_{2,4}$ with partition 
$\{a \wedge d, b \wedge c\}$ and $\{a \wedge b, a \wedge c, b \wedge d, c \wedge d\}$.
All edges are positive except for the ones connecting $b \wedge c$ with both $a \wedge b$ and $c \wedge d$.
See Figure \ref{fig:ext-c4}.
\vspace{.2in}

Our further interest will be on the exterior power $\wedgep{\Sigma}{k}$ of a signed graph $\Sigma = (G,\sigma)$.
This is a strict generalization of the case for unsigned graphs.
In this general case, the sign of the edge connecting $u,v$ in $\wedgep{\Sigma}{k}$ is 
the product of the sign of the permutation $\pi$ that connects $u$ and $v$, 
and the sign of the only adjacent pair $(u_{\pi(i)}, v_{i})$ in $\Sigma$, for $i$
which satisfies Equation (\ref{eqn:wedge-adjacency}); that is,
the sign of the edge $(u,v)$ in $\wedgep{\Sigma}{k}$ is
\begin{equation}
\sigma'(u,v) = \sgn(\pi) \times \sigma(u_{\pi(i)},v_{i}),
\end{equation}
where $(u_{\pi(i)},v_{i}) \in E(G)$ and $u_{\pi(j)} = v_{j}$ for all $j \neq i$.
Here, we have used $\sigma'$ to denote the signing function for the exterior power $\wedgep{\Sigma}{k}$.

\vspace{.1in}
\par\noindent{\bf Remark}:
A {\em gain} graph $\Phi = (G,X,\varphi)$ is a triple consisting of 
a graph $G=(V,E)$, a group $X$, and a gain function $\varphi$ over the edges of $G$,
that is, $\varphi: E(G) \rightarrow X$.
In our case, we may treat $\wedgep{G}{k}$ as a gain graph $\Phi(\wedgep{G}{k},S_{k},\varphi)$ 
over the symmetric group $S_{k}$, where $\varphi$ assigns the permutation $\pi$ that connects
$u$ and $v$ to the edge $(u,v)$.

%%%%%%%%%%%%%%%%%%%%%%%%%%%%%%%%%%%%%%%%%%%%%%%%%%%%%%%%%%%%%%
%% C4(Wedge 2)
%% by: C. Tamon (8/8/12)
%%%%%%%%%%%%%%%%%%%%%%%%%%%%%%%%%%%%%%%%%%%%%%%%%%%%%%%%%%%%%%
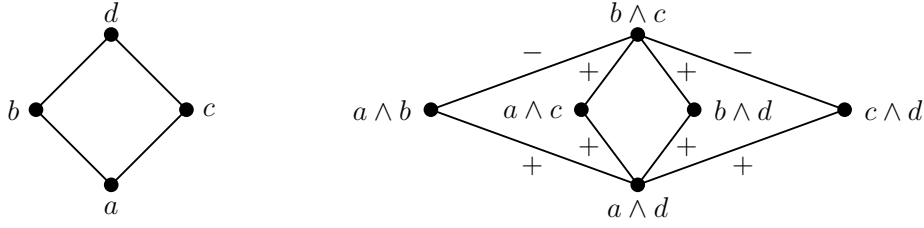
\begin{figure}[t]
\begin{center}
\begin{tikzpicture}

% C4
%
\draw[line width=0.25mm] (0,-1)--(-1,0)--(0,1)--(1,0)--(0,-1);

\foreach \y in {-1,1}
    \node at (0,\y)[circle, fill=black][scale=0.5]{};
\foreach \x in {-1,1}
    \node at (\x,0)[circle, fill=black][scale=0.5]{};

\node at (0,-1.3)[scale=0.9]{$a$};
\node at (-1.3,0)[scale=0.9]{$b$};
\node at (0,+1.3)[scale=0.9]{$d$};
\node at (+1.3,0)[scale=0.9]{$c$};

% signed K(2,4)
%
\foreach \y in {-1,+1}
	\foreach \x in {4.25,6.25,7.75,9.75}
		\draw[line width=0.25mm] (7,\y)--(\x,0);
\foreach \x in {4.25,6.25,7.75,9.75}
    \node at (\x,0)[circle, fill=black][scale=0.5]{};
\node at (3.6,0)[scale=0.9]{$a \wedge b$};
\node at (5.6,0)[scale=0.9]{$a \wedge c$};
\node at (8.4,0)[scale=0.9]{$b \wedge d$};
\node at (10.4,0)[scale=0.9]{$c \wedge d$};
   
\foreach \y in {-1,1}
	\node at (7,\y)[circle, fill=black][scale=0.5]{};
\node at (7,+1.3)[scale=0.9]{$b \wedge c$};
\node at (7,-1.3)[scale=0.9]{$a \wedge d$};

\node at (5.6,+0.75)[scale=0.9]{$-$};
\node at (5.6,-0.75)[scale=0.9]{$+$};
\node at (8.4,+0.75)[scale=0.9]{$-$};
\node at (8.4,-0.75)[scale=0.9]{$+$};

\node at (6.35,+0.5)[scale=0.9]{$+$};
\node at (6.35,-0.5)[scale=0.9]{$+$};
\node at (7.65,+0.5)[scale=0.9]{$+$};
\node at (7.65,-0.5)[scale=0.9]{$+$};
\end{tikzpicture}
\caption{The $4$-cycle and its exterior square $\wedgep{C_{4}}{2}$.
}
\label{fig:ext-c4}
\end{center}
\hrule
\end{figure}

%%%%%%%%%%%%%%%%%%%%%%%%%%%%%%%%%%%%%%%%%%%%%%%%%%%%%%%%%%%%%%%%%%%%%%%%%%%%%%%%%%%%%%%%%%%%%%%%%%%%%%%%%%%%%%%%%%
\paragraph{Algebraic framework}
Next, we describe an algebraic framework for the exterior $k$th power $\wedgep{G}{k}$.
Much of the machinery below will be familiar from multilinear algebra (see Bhatia \cite{bhatia}).

Let $G=(V,E)$ be a graph on $n$ vertices and let $k$ be a positive integer where $1 \le k \le n-1$.
We consider the vector space $W = \CC^{V}$ of dimension $n$ where we use the vertex set $V$ to 
index the $n$-dimensional space. 
%%% SWITCH OFF DIRAC %%%%%%%%%%%%%%%
\ignore{
\vspace{.1in}
\par\noindent{\bf Notation Alert}:\\
In what follows, we adopt the Dirac ``bra-ket'' notation which is convenient for our purposes 
(see \cite{nc00}). 
Using this notation, column vectors $\vec{z}$ in $\CC^{V}$ are written
as ``ket'' vectors $\ket{z}$ while row vectors $\vec{z}^{\dagger}$ are written as ``bra'' vectors 
$\bra{z}$ (with the understanding that $\bra{z} = \ket{z}^{\dagger}$). 
The notation supports the view that row vectors are linear functionals,
since $\bra{z}(\ket{w}) = \braket{z}{w}$ is simply the inner product of $\vec{z}$ and $\vec{w}$.
Similarly, $\ketbra{z}{w}$ is the outer product of $\vec{z}$ and $\vec{w}$.
\vspace{.1in}
}
%%%%%%%%%%%%%%%%%%%%%%%%%%%%%%%%%%%%
For each vertex $u \in V$, we let $\ket{u}$ represent the unit vector that is $1$ at position $u$ 
and $0$ elsewhere. So, the vertex set of $G$ gives us a basis for the vector space $W$ 
via the set $\{\ket{u} : u \in V\}$.
Next, we consider the $k$-fold tensor product space $\bigotimes^{k} W$ and index this 
$n^{k}$-dimensional space using the $k$-tuples of $V^{k}$. 
To this end, for a $k$-tuple $u = (u_{1},\ldots,u_{k}) \in V^{k}$, 
we define the unit vector $\ket{u}$ (this time in $\bigotimes^{k} W$) as
\begin{equation}
\ket{u}
=
\ket{u_{1}} \otimes \ket{u_{2}} \otimes \ldots \otimes \ket{u_{k}}.
\end{equation}
Moreover, the $k$-fold tensor product space $\bigotimes^{k} W$ is spanned by all such unit vectors.
For a permutation $\pi \in S_{k}$ and a $k$-tuple $u = (u_{1},\ldots,u_{k})$, 
the action of $\pi$ on $u$ is defined as $\pi(u) = (u_{\pi(1)},\ldots,u_{\pi(k)})$.
In a natural way, we may extend this to define the action of $\pi$ on the unit vector $\ket{u}$ as
\begin{equation}
\pi(\ket{u}) 
=
\ket{\pi(u)} 
= 
\ket{u_{\pi(1)}} \otimes \ket{u_{\pi(2)}} \otimes \ldots \otimes \ket{u_{\pi(k)}}.
\end{equation}
As before, we fix an arbitrary total ordering on the vertex set $V$, 
say $v_{1} \prec v_{2} \prec \ldots \prec v_{n}$. 
Then, we consider the set of ordered $k$-tuple of {\em distinct} elements of $V$, which is usually written
as ``wedge'' products:
\begin{equation}
\binom{V}{k} = \left\{ u_{1} \wedge u_{2} \wedge \ldots \wedge u_{k} : u_{1} \prec u_{2} \prec \ldots \prec u_{k}, 
					\mbox{ where } u_{1},\ldots,u_{k} \in V \right\}.
\end{equation}
Here, $u_{1} \wedge u_{2} \wedge \ldots \wedge u_{k}$ is notation for the $k$-tuple $(u_{1},u_{2},\ldots,u_{k})$
where $u_{1} \prec u_{2} \prec \ldots \prec u_{k}$.

The exterior power $\bigwedge^{k} W$ is a vector subspace of $\bigotimes^{k} W$
and we will use $\binom{V}{k}$ to index this $\binom{n}{k}$-dimensional space. 
Thus, we may identify $\bigwedge^{k} W$ with $\CC^{\binom{V}{k}}$ in the same way 
we identify $\bigotimes^{k} W$ with $\CC^{V^{k}}$.
The exterior power $\bigwedge^{k} W$ is spanned by the ``wedge'' vectors
\begin{equation}
\ket{u_{1} \wedge u_{2} \wedge \ldots \wedge u_{k}} 
=
\iverson{u_{1} \prec u_{2} \prec \ldots \prec u_{k}} \
\ket{u_{1}} \otimes \ket{u_{2}} \otimes \ldots \otimes \ket{u_{k}}.
\end{equation}
A nice way to view the connection between $\bigotimes^{k} W$ and $\bigwedge^{k} W$
is via the {\em anti-symmetrizer} $\Alt_{n,k}$ defined as the following $n^{k} \times \binom{n}{k}$ matrix:
\begin{equation}
\Alt_{n,k} 
= 
\frac{1}{\sqrt{k!}} \sum_{v \in \binom{V}{k}, \pi \in S_{k}} \sgn(\pi) \ketbra{\pi(v)}{v}.
\end{equation}
\ignore{
This shows that $\Alt_{n,k}$ is an injective map from $\bigwedge^{k} W$ into $\bigotimes^{k} W$ with
\begin{equation}
\Alt_{n,k}: \ket{v} \mapsto \frac{1}{\sqrt{k!}} \sum_{\pi} \sgn(\pi) \ket{\pi(v)}.
\end{equation}
Moreover, it also shows that $\Alt_{n,k}^{T}$ is a surjective map from $\bigotimes^{k} W$ to $\bigwedge^{k} W$ with
\begin{equation}
\Alt_{n,k}^{T}: \ket{u} \mapsto \frac{\sgn(\pi)}{\sqrt{k!}} \ket{v},
\end{equation}
where $v \in \binom{V}{k}$ and $\pi(u) = v$.
}

\vspace{.1in}
In what follows, we show the role of $\Alt_{n,k}$ in defining the exterior $k$th power $\wedgep{G}{k}$,
for a graph $G$. As a consequence, we note that $\wedgep{G}{k}$ is a signed quotient of the $k$-fold 
Cartesian product of $G$ modulo a natural equitable partition.

\begin{lemma} \label{lemma:exterior} %(see \cite{bgmrt})
Let $G$ be a graph on $n$ vertices and $k$ be an integer where $1 \le k \le n-1$.
The following properties hold:
\begin{enumerate}
\item $\Alt_{n,k}^{T}\Alt_{n,k} = I_{\binom{n}{k}}$.
\item $\Alt_{n,k} \Alt_{n,k}^{T}$ commutes with $A(G^{\Box k})$.
\item $A(\wedgep{G}{k}) = \Alt_{n,k}^{T} A(G^{\Box k}) \Alt_{n,k}$.
\end{enumerate}
\end{lemma}
\begin{proof}
See \cite{bgmrt}.
\end{proof}

\par\noindent
We note that Lemma \ref{lemma:exterior} is similar to Fact \ref{fact:equitable-partition}.
So from the third property above, $\wedgep{G}{k}$ may be viewed as the ``quotient'' of 
$G^{\Box k}$ modulo an equitable partition $\pi$ induced by $\Alt_{n,k}$; 
that is, $\wedgep{G}{k} = G^{\Box k}/\pi$.

%%%%%%%%%%%%%%%%%%%%%%%%%%%%%%%%%%%%%%%%%%%%%%%%%%%%%%%%%%%%%%%%%%%%%%%%%%%%%%%%%%%%%%%%%%%%%%%%%%%%%%%%%%%%%%%%%%

\section{Basic Properties}

We start with the simplest fact about the exterior $k$th power: cases when $k=1$ and $k=n$.

\begin{fact}
For any $n$-vertex signed graph $\Sigma$, we have $\wedgep{\Sigma}{1} = \Sigma$,
and $\wedgep{\Sigma}{n} = K_{1}$.
\end{fact}
\begin{proof}
Follows immediately from the definitions.
\end{proof}

Next, we show that the exterior power $\wedgep{G}{k}$ is {\em well-defined} up to switching equivalence; 
that is, the signed graph $\wedgep{G}{k}$ we obtain is independent of the choice of ordering on the
vertex set $V$ of $G$.
This allows us to simply state $\wedgep{G}{k}$ without fear of ambiguity.
To this end, let $\wedgepo{G}{k}{\prec}$ denote the signed graph obtained from ordering
$V$ according to $\prec$. 

\begin{fact} \label{fact:ordering-invariant}
Let $G=(V,E)$ be a graph.
For any two total orderings $\prec$ and $\prec'$ over $V$, 
we have $\wedgepo{G}{k}{\prec} \sim \wedgepo{G}{k}{\prec'}$.
\end{fact}
\begin{proof}
The total ordering $\prec$ determines the map $\Alt^{\prec}_{n,k}$.
Let $\prec'$ be a total ordering which differs from $\prec$ by a transposition on two vertices $u$ and $v$.
In a similar manner, the total ordering $\prec'$ determines its own anti-symmetrizer  $\Alt^{\prec'}_{n,k}$.
Define a diagonal matrix $D_{uv}$ with $\pm 1$ entries where, for each $a \in \binom{V}{k}$, 
its $(a,a)$-entry is
\begin{equation}
\bra{a}D_{uv}\ket{a} =
	%(-1)^{\iverson{\mbox{\scriptsize $u \in a$ and $v \in a$}}} =
	\left\{\begin{array}{ll}
	-1 & \mbox{ if $u \in a$ and $v \in a$ } \\
	+1 & \mbox{ otherwise }
	\end{array}\right.
\end{equation}
The two anti-symmetrizers of $\prec$ and $\prec'$ are connected by $D_{uv}$ as follows:
\begin{equation} \label{eqn:order-switching}
\Alt_{n,k}^{\prec'}
%	& = & \frac{1}{k!} \sum_{a \in \binom{V'}{k}, \pi \in S_{k}} 
%		\sgn(\pi) P_{\pi}\ketbra{a}{a} \\
%	& = & \frac{1}{k!} \sum_{a \in \binom{V}{k}, \pi \in S_{k}} 
%		\bra{a}D_{uv}\ket{a} \sgn(\pi) P_{\pi} \ketbra{a}{a} \\
	= \Alt_{n,k}^{\prec} D_{uv}.
\end{equation}
This shows that $\wedgepo{G}{k}{\prec}$ are switching equivalent to $\wedgepo{G}{k}{\prec'}$ since
\begin{eqnarray}
A(\xwedgepo{G}{k}{\prec'}) 
	& = & (\Alt_{n,k}^{\prec'})^{T} A(G^{\Box k}) \Alt_{n,k}^{\prec'} \\
%		\ \ \ \mbox{ by definition of $A(\wedgepo{G}{k}{\prec'})$} \\
	& = & D_{uv}^{T} (\Alt_{n,k}^{\prec})^{T} A(G^{\Box k}) \Alt_{n,k}^{\prec} D_{uv},
		\ \ \ \mbox{ by Equation (\ref{eqn:order-switching})} \\
	& = & D_{uv}^{-1} A(\xwedgepo{G}{k}{\prec}) D_{uv},
%		\ \ \ \mbox{ by definition of $A(\wedgepo{G}{k}{\prec})$ and since $D_{uv}^{T} = D_{uv}^{-1}$.}
		\ \ \ \mbox{ since $D_{uv}^{T} = D_{uv}^{-1}$.}
\end{eqnarray}
In general, if $\prec$ and $\prec'$ are related by a permutation $\sigma$,
we may apply induction on the number of transpositions that form $\sigma$.
\end{proof}

The next fact shows that there is a mirror-symmetric isomorphism in the sequence of
exterior powers $\wedgep{G}{k}$.

\begin{fact} \label{fact:palindrome-isomorphic}
Let $G$ be a graph on $n$ vertices and let $1 \le k \le n-1$. 
Then, $|\wedgep{G}{k}| \cong |\wedgep{G}{n-k}|$.
\end{fact}
\begin{proof}
The map $\tau$ from $\binom{V}{k}$ to $\binom{V}{n-k}$ defined as $\tau(u) = V(G) \setminus u$
is the claimed isomorphism.
If $u$ is adjacent to $v$ in $|\wedgep{G}{k}|$, then $u \symdif v = \{a,b\}$ for some vertices $a$ and $b$
with $(a,b) \in E(G)$. 
Since $(V(G) \setminus u) \symdif (V(G) \setminus v) = u \symdif v$, this implies $\tau(u)$ is adjacent to 
$\tau(v)$ in $|\wedgep{G}{n-k}|$. The converse follows in a similar fashion.
\end{proof}

\par\noindent{\em Remark}:
Fact \ref{fact:palindrome-isomorphic} was also described in Osborne \cite{o06}; we provide its proof
for completeness. A similar mirror-symmetry condition is also known for Johnson graphs 
(see Godsil and Royle \cite{godsil-royle01}).

\vspace{.1in}
Next, we show that the exterior power, as a graph operator on signed graphs, preserves switching equivalence.

\begin{fact} \label{fact:switching-class-preserving}
Let $\Sigma_{1}$ and $\Sigma_{2}$ be two signed graphs on $n$ vertices satisfying $\Sigma_{1} \sim \Sigma_{2}$. 
Then, $\wedgep{\Sigma_{1}}{k} \sim \wedgep{\Sigma_{2}}{k}$ for each positive integer $k$ with $1 \le k \le n-1$.
\end{fact}
\begin{proof}
If $\Sigma_{1} \sim \Sigma_{2}$, then there is a diagonal $n \times n$ matrix $D$ with $\pm 1$ entries
so that $A(\Sigma_{2}) = DA(\Sigma_{1})D$. This implies 
$A(\Sigma_{2}^{\Box k}) = D^{\otimes k} A(\Sigma_{1}^{\Box k}) D^{\otimes k}$.
Let $\hat{D}$ be a diagonal matrix with $\pm 1$ entries of dimension $\binom{n}{k}$ where,
for each $u \in \binom{V}{k}$, its $(u,u)$-entry is defined as
\begin{equation}
\bra{u}\hat{D}\ket{u} = \prod_{j=1}^{k} \bra{u_{j}}D\ket{u_{j}},
\end{equation}
which depends on the product of the diagonal entries of $D$.
We note that
\begin{equation}
\Alt_{n,k} \hat{D} = D^{\otimes k} \Alt_{n,k}.
\end{equation}
Therefore, we have
\begin{eqnarray}
A(\xwedgep{\Sigma_{2}}{k}) 
	& = & \Alt_{n,k}^{T}A(\Sigma_{2}^{\Box k})\Alt_{n,k} \\
	& = & \Alt_{n,k}^{T} D^{\otimes k} A(\Sigma_{1}^{\Box k}) D^{\otimes k} \Alt_{n,k} \\
	& = & \hat{D} \ \Alt_{n,k}^{T} A(\Sigma_{1}^{\Box k}) \Alt_{n,k} \hat{D} \\
	& = & \hat{D} A(\xwedgep{\Sigma_{1}}{k}) \hat{D},
\end{eqnarray}
which proves the claim.
\end{proof}

%%%%%%%%%%%%%%%%%%%%%%%%%%%%%%%%%%%%%%%%%%%%%%%%%%%%%%%%%%%%%%%%%%%%%%%%%%%%%%%%%%%%%%%%%%%%%%%%%%%%%%%%%%%%%%%%%%

In the following, we consider the $k$-fold Cartesian product graph $G^{\Box k}$ minus 
its ``diagonal'' (which are the $k$-tuples containing a repeated element), but without the
conjugation (or quotient) action. We relate this to the gain graphs and signed graphs 
corresponding to $\wedgep{G}{k}$.
Note that there is a natural way to convert a gain graph over the symmetric group into a signed graph --
simply take the sign of the permutations which label the edges of the gain graph.

But, first we recall some definitions about covers of gain graphs.
Let $\Phi = (G,X,\varphi)$ be a gain graph on a graph $G = (V,E)$ and a group $X$ 
with the group-valued edge-signing function $\varphi: E \rightarrow X$.
The graph $\SC$ is called a $|X|$-cover of $\Phi$ if the adjacency matrix of $\SC$ is given by
\begin{equation}
A(\SC) = \sum_{g \in X} A(G^{g}) \otimes P_{g}.
\end{equation}
Here, $G^{g}$ is the subgraph of $G$ induced by edges which are signed with $g \in X$; 
that is, $G^{g} = (V,E^{g})$ where $E^{g} = \{e \in E : \varphi(e) = g\}$.
Also, $P_{g}$ is the $|X|$-dimensional permutation matrix induced by $g$;
its entries are given by $P_{g}[a,b] = \iverson{a = gb}$, for group elements $a,b \in X$.

\vspace{.1in}

\begin{fact} \label{fact:gain-cover}
Let $G = (V,E)$ be a graph on $n$ vertices and let $k$ be a positive integer $1 \le k < n$. 
Then, $G^{\Box k} \setminus \Ker(\Alt_{n,k}^{T})$ is a $|S_{k}|$-cover of the gain graph 
$\Phi(\wedgep{G}{k},S_{k},\varphi)$. 
\end{fact}
\begin{proof}
Let $\SC$ be the $|S_{k}|$-cover of $\Phi(\wedgep{G}{k},S_{k},\varphi)$.
Each vertex of $\SC$ is of the form $(u,\pi)$ where $u \in \binom{V}{k}$ and $\pi \in S_{k}$.
Let $D = \Ker(\Alt_{n,k}^{T})$ and consider the map $\phi: \SC \rightarrow G^{\Box k} \setminus D$ 
defined as $\phi(u,\pi) = \pi(u)$. 
We claim that $\phi$ is an isomorphism between $\SC$ and $G^{\Box k} \setminus D$.
Suppose $(u,\pi)$ and $(v,\tau)$ is adjacent in $\SC$. 
This implies $u$ and $v$ are adjacent in $G^{\wedge k}$ and $\tau = \pi \circ \varphi(u,v)$.
Therefore, $\pi(u)$ is adjacent to $\tau(v) = \varphi(u,v)(\pi(v))$ in $G^{\Box k} \setminus D$
from the definition of the gain graph $\Phi(\wedgep{G}{k},S_{k},\varphi)$.
\end{proof}

\vspace{.1in}

\begin{corollary} \label{cor:double-cover}
Let $G$ be a graph with at least three vertices and let $D = \{a \otimes a : a \in V(G)\}$.
Then, $G^{\Box 2} \setminus D$ is a {double}-cover of the signed graph $\wedgep{G}{2}$. 
\end{corollary}

%%%%%%%%%%%%%%%%%%%%%%%%%%%%%%%%%%%%%%%%%%%%%%%%%%%%%%%%%%%%%%%%%%%%%%%%%%%%%%%%%%%%%%%%%%%%%%%%%%%%%%%%%%%%%%%%%%

\section{Unsigned Graphs}

To provide an illustration of upcoming techniques we used for signed graphs,
we examine the exterior powers of some well-known unsigned families of graphs,
such as, cliques, paths, cycles, and hypercubes.

\subsection{Cliques}

Given positive integers $n \ge k \ge \ell \ge 1$, the graph $J(n,k,\ell)$ has vertices which are $k$-element 
subsets of $[n]$ and edges between two subsets if their intersection is of size $\ell$ 
(see Godsil and Royle \cite{godsil-royle01}). The family of graphs $J(n,k,k-1)$ is known as the {\em Johnson} 
graphs.

\vspace{.1in}

\begin{fact} \label{fact:johnson}
For positive integers $n$ and $k$ where $1 \le k \le n-1$, 
we have $|\wedgep{K_{n}}{k}| \cong J(n,k,k-1)$. 
\end{fact}
\begin{proof}
Since any two distinct vertices of $K_{n}$ are adjacent, the exterior power
$\wedgep{K_{n}}{k}$ has each $k$-subset $A$ of $V$ adjacent to every other $k$-subset $B$ of $V$
provided $|A \symdif B| = 2$. This is exactly the definition of the Johnson graphs $J(n,k,k-1)$.
\end{proof}

\noindent
The above fact also appeared in \cite{o06}; we include its proof for completeness.

\subsection{Paths}

\begin{fact} \label{fact:path-exterior}
For positive integers $n$ and $k$ where $1 \le k \le n-1$, 
the exterior power of a path $\wedgep{P_{n}}{k}$ is always balanced.
\end{fact}
\begin{proof}
Let the vertex set of $P_{n}$ be $V = \{a_{0},\ldots,a_{n-1}\}$ with $a_{i}$ is adjacent to $a_{j}$
whenever $i - j = \pm 1$ except that $a_{0}$ is only connected to $a_{1}$ and $a_{n-1}$ is only
connected to $a_{n-2}$. By Fact \ref{fact:ordering-invariant}, we may use the ordering $a_{i} \prec a_{j}$ 
whenever $i < j$. 
Consider an arbitrary cycle of length $\ell$ in $\wedgep{P_{n}}{k}$ that is given by the 
sequence of vertices $w_{1},\ldots,w_{\ell}$, where $w_{i} \in \binom{V}{k}$.
The permutation that connects $w_{i}$ and $w_{i+1}$ is always the identity permutation 
since $w_{i} \symdif w_{i+1} = \{a_{j},a_{j+1}\}$ for some $j$. 
Thus, the sign of all edges in $\wedgep{P_{n}}{k}$ is always positive.
\end{proof}

\subsection{Cycles}

\begin{fact} \label{fact:cycle-exterior}
For positive integers $n$ and $k$ where $1 \le k \le n-1$, 
the exterior power of a cycle $\wedgep{C_{n}}{k}$ is balanced if and only if $k$ is odd.
\end{fact}
\begin{proof}
Let the vertex set of $C_{n}$ be $V = \{a_{0},\ldots,a_{n-1}\}$ with $a_{i}$ is adjacent to $a_{j}$
whenever $i - j \equiv \pm 1\pmod{n}$. 
By Fact \ref{fact:ordering-invariant}, we may use the ordering $a_{i} \prec a_{j}$ whenever $i < j$. 
Consider an arbitrary cycle of length $\ell$ in $\wedgep{C_{n}}{k}$ that is given by the 
sequence of vertices $w_{1},\ldots,w_{\ell}$, where $w_{i} \in \binom{V}{k}$.
The permutation that connects $w_{i}$ and $w_{i+1}$ is a $k$-cycle 
if $w_{i} \symdif w_{i+1} = \{a_{0},a_{n-1}\}$, otherwise it is the identity permutation.
Since a $k$-cycle has a negative sign whenever $k$ is even,
this shows $\wedgep{C_{n}}{k}$ is balanced if and only if $k$ is odd.
\end{proof}

\vspace{.1in}
We will later generalize Facts \ref{fact:path-exterior} and \ref{fact:cycle-exterior}
for signed graphs in Section \ref{section:signed-balanced}.

\subsection{Hypercubes}

\begin{fact}
Let $n \ge 3$ be a positive integer and let $k$ be an integer where $1 \le k \le N-1$
and $N=2^{n}$. Then, $\bigwedge^{k} Q_{n}$ is a signed bipartite graph and
$\bigwedge^{k} Q_{n}$ is unbalanced except for $k=1,N-1$.
\end{fact}
\begin{proof}
To see that $|\bigwedge^{k} Q_{n}|$ is bipartite, note that we may identify the vertices of 
$Q_{n}$ with the set $V = \zo^{n}$ of binary strings of length $n$. There is a natural bipartition 
of $V$ into $V = W_{0} \uplus W_{1}$ according to the Hamming weight of the strings, where 
$W_{0}$ and $W_{1}$ are the binary strings with even and odd Hamming weight, respectively.
Let $V_{\ell}$ be the collection of $k$-subsets of $V$ which contain $\ell$ elements from
$W_{0}$ and $k-\ell$ elements from $W_{1}$.
Then, we have a bipartition $\wedgep{Q_{n}}{k} = V_{0} \uplus V_{1}$ where
$V_{0} = \biguplus \{V_{\ell} : \ell \mbox{ even}\}$
and $V_{1} = \biguplus \{V_{\ell} : \ell \mbox{ odd}\}$.
Note, each vertex in $V_{0}$ is connected to only vertices from $V_{1}$, and vice versa.

By Lemma \ref{lemma:claw-free}, $\bigwedge^{k} Q_{n}$ is unbalanced, for $k \neq 1,N-1$, 
since $Q_{n}$ contains the claw $K_{1,3}$ as a subgraph whenever $n \ge 3$. 
Note $\wedgep{Q_{n}}{1} = Q_{n}$ is trivially balanced and $\bigwedge^{N-1} Q_{n}$ is balanced 
since $Q_{n}$ is bipartite (see Corollary \ref{cor:unsigned-balanced}).
\end{proof}

%%%%%%%%%%%%%%%%%%%%%%%%%%%%%%%%%%%%%%%%%%%%%%%%%%%%%%%%%%%%%%%%%%%%%%%%%%%%%%%%%%%%%%%%%%%%%%%%%%%%%%%%%%%%%%%%%%

\section{Balanced Characterization} \label{section:signed-balanced}

In this section, we characterize those {\em signed} graphs $\Sigma$ for which $\wedgep{\Sigma}{k}$ 
is {\em balanced}; that is, signed graphs $\Sigma$ which satisfy $\wedgep{\Sigma}{k} \sim |\wedgep{\Sigma}{k}|$.
Our main result is Theorem \ref{thm:signed-balanced}, but first we prove several preliminary lemmas. 
The next lemma shows that if $\wedgep{\Sigma}{k}$ is balanced, then $|\Sigma|$ does not contain the 
claw $K_{1,3}$ as a subgraph; this holds if $k$ is strictly smaller than $|V(\Sigma)|-1$.

\begin{lemma} \label{lemma:claw-free}
Let $\Sigma$ be a signed graph on $n$ vertices.
If $K_{1,3}$ is a subgraph of $|\Sigma|$, then $\wedgep{\Sigma}{k}$ is not balanced, for
$2 \le k \le n-2$.
\end{lemma}
\begin{proof}
Let $\Sigma = (G,\sigma)$ be a signed graph where $G = (V,E)$ contains $K_{1,3}$ as a subgraph.
By switching around the center of the claw, we may assume that the subgraph $K_{1,3}$ is all positive. 
Suppose the vertices of the claw is $V_{1} = \{a,b,c,d\}$, where $a$ is the center of the claw, 
and that $V = V_{1} \uplus V_{2}$, where $V_{2}$ is the set of vertices of $G$ that is not part of the claw.
By Fact \ref{fact:ordering-invariant}, we may assume a total ordering on $V$ which starts with the
vertices of $V_{1}$ followed by the vertices of $V_{2}$ (in some arbitrary order):
\begin{equation}
a \prec b \prec c \prec d \prec \ldots
\end{equation}
Consider the cycle on six vertices in $\wedgep{K_{1,3}}{k}$ defined by the following sequence: 
\begin{equation}
a \wedge b \wedge \Omega, 
\ \ 
b \wedge c \wedge \Omega, 
\ \ 
a \wedge c \wedge \Omega, 
\ \ 
c \wedge d \wedge \Omega, 
\ \ 
a \wedge d \wedge \Omega, 
\ \ 
b \wedge d \wedge \Omega,
\end{equation}
where $\Omega$ is the wedge product of some arbitrary $k-2$ elements from $V_{2}$.
Note the above sequence induces a negative $C_{6}$ (see Figure \ref{fig:ext-claw2}).
\end{proof}

%%%%%%%%%%%%%%%%%%%%%%%%%%%%%%%%%%%%%%%%%%%%%%%%%%%%%%%%%%%%%%
%% K(1,3)(Wedge 2)
%% by: C. Tamon (8/9/12)
%%%%%%%%%%%%%%%%%%%%%%%%%%%%%%%%%%%%%%%%%%%%%%%%%%%%%%%%%%%%%%
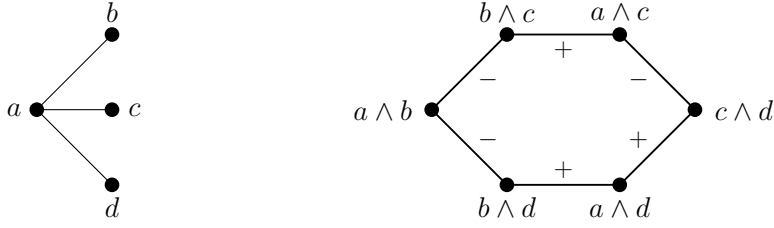
\begin{figure}[t]
\begin{center}
\begin{tikzpicture}

% K(1,3)
%
\foreach \y in {-1,0,1} {
	\draw (-1,0)--(0,\y);
    \node at (0,\y)[circle, fill=black][scale=0.5]{};
}
\node at (-1,0)[circle, fill=black][scale=0.5]{};

\node at (-1.3,0)[scale=0.9]{$a$};
\node at (0,1.3,0)[scale=0.9]{$b$};
\node at (0.3,0)[scale=0.9]{$c$};
\node at (0,-1.3)[scale=0.9]{$d$};

% C6
%
\draw[line width=0.25mm] (4.25,0)--(5.25,1)--(6.75,1)--(7.75,0)--(6.75,-1)--(5.25,-1)--(4.25,0);

\foreach \x in {4.25,7.75}
	\node at (\x,0)[circle, fill=black][scale=0.5]{};
\node at (3.6,0)[scale=0.9]{$a \wedge b$};
\node at (8.4,0)[scale=0.9]{$c \wedge d$};

\foreach \x in {5.25,6.75}
	\foreach \y in {-1,+1}
		\node at (\x,\y)[circle, fill=black][scale=0.5]{};
\node at (5.25,+1.3)[scale=0.9]{$b \wedge c$};
\node at (6.75,+1.3)[scale=0.9]{$a \wedge c$};
\node at (6.75,-1.3)[scale=0.9]{$a \wedge d$};
\node at (5.25,-1.3)[scale=0.9]{$b \wedge d$};

\node at (5,+0.4)[scale=0.8]{$-$};
\node at (5,-0.4)[scale=0.8]{$-$};
\node at (7,+0.4)[scale=0.8]{$-$};
\node at (7,-0.4)[scale=0.8]{$+$};
\node at (6,+0.8)[scale=0.8]{$+$};
\node at (6,-0.8)[scale=0.8]{$+$};
\end{tikzpicture}
\caption{The claw $K_{1,3}$ and its exterior square which is a negative $C_{6}$.
}
\label{fig:ext-claw2}
\end{center}
\hrule
\end{figure}

\vspace{.1in}
Next, we show that the exterior $k$th power of signed cycles are balanced provided 
the cycle is balanced and $k$ is odd or it is unbalanced and $k$ is even.
Note, this implies that the exterior $k$th power of signed paths are always balanced.

\begin{lemma} \label{lemma:signed-cycle}
Let $\Sigma$ be a signed cycle with $n$ vertices and let $k$ be an integer with $1 \le k \le n-1$.
Then, $\wedgep{\Sigma}{k}$ is balanced if and only if 
either
(i) $\Sigma$ is balanced and $k$ is odd; 
or
(ii) $\Sigma$ is unbalanced and $k$ is even.
\end{lemma}
\begin{proof}
Let $\Sigma = (C_{n},\sigma)$ be a signed cycle. 
If $\Sigma$ is unbalanced, then it has an odd number of negative edges. 
Since two adjacent negative edges can be switched to positive edges
and switching may move any negative edge along the cycle,
we may assume that there is exactly one negative edge on $\Sigma$.
If $\Sigma$ is balanced, we may assume (by switching) that 
it has no negative edges.
So, $\Sigma$ either has no (if balanced) or exactly one negative edge (if unbalanced).

We assume without loss of generality that the vertices of $C_{n}$ is
$V = \{u_{0},u_{1},\ldots,u_{n-1}\}$ where $u_{i}$ is adjacent to $u_{j}$ 
if $i-j \equiv \pm 1\pmod{n}$ and so that
$(u_{0},u_{n-1})$ is the unique negative edge if $\Sigma$ is not balanced.
By Fact \ref{fact:ordering-invariant}, we may assume that the ordering 
on $V$ is $u_{0} \prec u_{1} \prec \ldots \prec u_{n-1}$.

Consider an arbitrary cycle $\Gamma$ in $\wedgep{\Sigma}{k}$ given by 
the sequence $w_{1},w_{2},\ldots,w_{\ell} \in \binom{V}{k}$.
The permutation that connects $w_{i}$ and $w_{i+1}$, which we will denote
$\pi_{w_{i},w_{i+1}}$, is a $k$-cycle if $w_{i} \symdif w_{i+1} = \{u_{0},u_{n-1}\}$, 
and is the identity permutation otherwise.
Let $\sigma'$ denote the edge signing function on $\wedgep{\Sigma}{k}$.
Thus, the sign of the edge $(w_{i},w_{i+1})$ is given by
\begin{eqnarray}
\sigma'(w_{i},w_{i+1})
	& = & \sigma(w_{i} \symdif w_{i+1}) \sgn(\pi_{w_{i},w_{i+1}}) \\
	& = &
		\left\{\begin{array}{ll}
		(-1)^{\iverson{\mbox{\scriptsize $\Sigma$ unbalanced}}} \sgn(\mbox{\footnotesize $k$-cycle}) & \mbox{ if $w_{i} \symdif w_{i+1} = \{u_{0},w_{n-1}\}$ } \\
		(+1) \sgn(\mbox{\footnotesize id}) & \mbox{ if $w_{i} \symdif w_{i+1} \neq \{u_{0},w_{n-1}\}$ }
		\end{array}\right.
\end{eqnarray}
Since $\sgn(\mbox{\small $k$-cycle}) = (-1)^{k-1}$, the sign of each edge $(w_{i},w_{i+1})$ of $\Gamma$ 
is positive if $\Sigma$ is unbalanced and $k$ is even or if $\Sigma$ is balanced and $k$ is odd.
Thus, each cycle $\Gamma$ in $\wedgep{\Sigma}{k}$ is a positive 
if and only if 
$\Sigma$ is unbalanced and $k$ is even or $\Sigma$ is balanced and $k$ is odd.
\end{proof}

\vspace{.1in}

\begin{lemma} \label{lemma:signed-path}
Let $\Sigma$ be a signed path on $n$ vertices.
Then, $\wedgep{\Sigma}{k}$ is balanced for each integer $k$ with $1 \le k \le n-1$.
\end{lemma}
\begin{proof}
Note that $\Sigma$ is balanced since $|\Sigma|$ is a path so we may assume that all edges
are positive. This already proves the case for $k=1$, since $\wedgep{\Sigma}{1} = \Sigma$,
and for $k=n-1$, since $|\wedgep{\Sigma}{n-1}| \cong |\Sigma|$.
For the cases $2 \le k \le n-2$, we appeal to the proof of Lemma \ref{lemma:signed-cycle}.
In an arbitrary cycle $\Gamma$ in $\wedgep{\Sigma}{k}$, the permutation connecting
adjacent vertices in $\Gamma$ is always the identity permutation. That is, we never
encounter the case where the symmetric difference between adjacent vertices
is the two endpoints of the path $\Sigma$.
This shows that any such cycle $\Gamma$ in $\wedgep{\Sigma}{k}$ is always positive.
\end{proof}

\vspace{.1in}
Next, we characterize the $(n-1)$th exterior power of $n$-vertex signed graphs.
In this case, we show that only balanced signed bipartite graphs have a balanced 
$(n-1)$th exterior power.

\begin{lemma} \label{lemma:signed-bipartite}
Let $\Sigma$ be a signed graph on $n$ vertices.
The exterior power $\wedgep{\Sigma}{n-1}$ is balanced if and only if 
each odd cycle in $\Sigma$ is negative and each even cycle in $\Sigma$ is positive.
\end{lemma}
\begin{proof}
Let $\Sigma = (G,\sigma)$ be a signed graph where $G=(V,E)$ and $|V|=n$.
Suppose $\theta: \binom{V}{n-1} \rightarrow V$ is the natural bijection between $V$
and its $(n-1)$-subsets defined as follows.
For a subset $A \subseteq V$ with $|A| = n-1$, we let
$\theta(A) = a$ where $V \setminus \{a\} = A$.

Consider an arbitrary cycle $\Gamma$ of length $\ell$ in $\wedgep{\Sigma}{n-1}$ given by the sequence
$w_{0},w_{1},\ldots,w_{\ell-1} \in \binom{V}{n-1}$.
We identify $\Gamma$ with another cycle $\theta(\Gamma)$ in $\Sigma$ 
defined by the sequence 
\begin{equation}
\theta(w_{0}), \theta(w_{1}), \ldots, \theta(w_{\ell-1}) \in V.
\end{equation}
By Fact \ref{fact:ordering-invariant}, we may assume an ordering on $V$ which
starts with vertices in $\theta(\Gamma)$ followed by the other vertices (in some arbitrary order):
\begin{equation} \label{eqn:bipartite-ordering}
\theta(w_{0}) \prec \theta(w_{1}) \prec \ldots \prec \theta(w_{\ell-1}) \prec \ldots
\end{equation}
By switching, we may assume that $\theta(\Gamma)$ has zero or exactly one negative edge,
depending on whether $\theta(\Gamma)$ is a positive or negative cycle. 
Moreover, without loss of generality, we may assume that that the negative edge, if it exists, 
be the edge connecting $\theta(w_{0})$ and $\theta(w_{\ell-1})$.
That is,
\begin{eqnarray}
\sigma(\theta(w_{i}), \theta(w_{i+1}))
	& = & +1,		\ \ \ \mbox{ for $i = 0,\ldots,\ell-2$ } \\
\sigma(\theta(w_{0}), \theta(w_{\ell-1})) 	
	%& = & (-1)^{\iverson{\mbox{\scriptsize $\theta(\Gamma)$ negative}}}.
	& = & 
		\left\{\begin{array}{ll}
		-1 & \mbox{ if $\theta(\Gamma)$ is a negative cycle} \\
		+1 & \mbox{ otherwise }
		\end{array}\right.
\end{eqnarray}
We use $\sigma'(w_{i},w_{i+1})$ to denote the sign of edge $(w_{i},w_{i+1})$ in $\wedgep{\Sigma}{n-1}$.
Note that 
\begin{equation}
\sigma'(w_{i},w_{i+1}) = \sigma(\theta(w_{i}), \theta(w_{i+1})).
\end{equation}
Let $\pi_{w_{i},w_{i+1}}$ denote the permutation that connects the vertices
$w_{i}$ and $w_{i+1}$ in $\wedgep{\Sigma}{n-1}$.
By our choice of the ordering given above,
the identity permutation connects $(w_{i},w_{i+1})$, for $0 \le i \le \ell-2$,
and an $(\ell-1)$-cycle connects $(w_{0},w_{\ell-1})$.
Therefore, the sign of the cycle $\Gamma$ is given by
\begin{eqnarray} \label{eqn:cycle-sign}
\sigma'(\Gamma) 
	& = & \sigma'(w_{0}, w_{\ell-1}) \sgn(\pi_{w_{0},w_{\ell-1}})
		\prod_{i=0}^{\ell-2} \sigma'(w_{i}, w_{i+1}) \sgn(\pi_{w_{i},w_{i+1}}) \\
	& = & \sigma(\theta(w_{0}), \theta(w_{\ell-1})) \sgn(\pi_{w_{0},w_{\ell-1}})
		\prod_{i=0}^{\ell-2} \sigma(\theta(w_{i}), \theta(w_{i+1})) \sgn(\pi_{w_{i},w_{i+1}}) \\
	& = & (-1)^{\iverson{\mbox{\scriptsize $\theta(\Gamma)$ negative}}} \sgn(\mbox{\small $(\ell-1)$-cycle}).
\end{eqnarray}
The last expression equals $+1$ if and only if 
each odd cycle in $\Sigma$ is negative and each even cycle in $\Sigma$ is positive.
This proves the claim.
\end{proof}

%%%%%%%%%%%%%%%%%%%%%%%%%%%%%%%%

\vspace{.1in}
Finally, we are ready to put the pieces together and prove our main characterization 
of signed graphs $\Sigma$ whose exterior power $\wedgep{\Sigma}{k}$ is balanced.
\vspace{.1in}

\ignore{
\begin{theorem} \label{thm:signed-balanced}
Let $\Sigma$ be a signed graph on $n$ vertices. Then, $\wedgep{\Sigma}{k}$ is balanced if and only if:
\begin{itemize}
\item for $k = 1,\ldots,n-2$:\\
	Either (i) $|\Sigma|$ is a path; or
	(ii) $|\Sigma|$ is a cycle, where $\Sigma$ is balanced and $k$ is odd 
	or $\Sigma$ is unbalanced and $k$ is even.
\item for $k=n-1$:\\
	Each even cycle in $\Sigma$ is positive and each odd cycle in $\Sigma$ is negative.
\end{itemize}
\end{theorem}
}
\begin{proof} (of Theorem \ref{thm:signed-balanced}) \\
The first part follows from Lemmas \ref{lemma:claw-free}, \ref{lemma:signed-path}, and \ref{lemma:signed-cycle}.
The second part is a restatement of Lemma \ref{lemma:signed-bipartite}.
\end{proof}

%%%%%%%%%%%%%%%%%%%%%%%%%%%%%%%%

\vspace{.1in}
The next corollary of Theorem \ref{thm:signed-balanced} yields a similar characterization 
for {\em unsigned} graphs.

\begin{corollary} \label{cor:unsigned-balanced}
Let $G$ be a graph on $n$ vertices.
Then, $\wedgep{G}{k}$ is balanced if and only if:
\begin{itemize}
\item for $k=1,\ldots,n-2$: $G$ is a path or $G$ is a cycle and $k$ is odd.
\item for $k=n-1$: $G$ is bipartite.
\end{itemize}
\end{corollary}

%%%%%%%%%%%%%%%%%%%%%%%%%%%%%%%%%%%%%%%%%%%%%%%%%%%%%%%%%%%%%%%%%%%%%%%%%%%%%%%%%%%%%%%%%%%%%%%%%%%%%%%%%%%%%%%%%%

\section*{Acknowledgments}

The research was supported in part by the National Science Foundation grant DMS-1004531
and also by the National Security Agency grant H98230-11-1-0206.
We thank Chris Godsil for helpful comments on the symmetric powers of graphs.

%%%%%%%%%%%%%%%%%%%%%%%%%%%%%%%%%%%%%%%%%%%%%%%%%%%%%%%%%%%%%%%%%%%%%%%%%%%%%%%%%%%%%%%%%%%%%%%%%%%%%%%%%%%%%%%%%%

%\bibliography{qwalk}

%%%%%%%%%%%%%%%%%%%%%%%%%%%%%%%%%%%%%%%%%%%%%%%%%%%%%%%%%%%%%%%%%%%%%%%%%%%%%%%%%%%%%%%%%%%%%%%%%%%%%%%%%%%%%%%%%%

\end{document}